\documentclass[lettersize,journal]{IEEEtran}
\usepackage{amsmath,amsfonts}
\usepackage{algorithmic}
\usepackage{algorithm}
\usepackage{array}
\usepackage[caption=false,font=normalsize,labelfont=sf,textfont=sf]{subfig}
\usepackage{textcomp}
\usepackage{stfloats}
\usepackage{url}
\usepackage{verbatim}
\usepackage{graphicx}
\usepackage{cite}
\usepackage{colortbl}

\newtheorem{theorem}{Theorem}

\newtheorem{remark}{Remark}
\newtheorem{definition}{Definition}
\newtheorem{assumption}{Assumption}
\newtheorem{lemma}{Lemma}

\newtheorem{proof}{Proof}
\allowdisplaybreaks[4]

\begin{document}

\title{Private Inputs for Leader-Follower Game with Feedback Stackelberg Strategy}

\author{Yue Sun, Hongdan Li and Huanshui Zhang*
\thanks{This work was supported by the National Natural Science Foundation of China under Grants 61821004 and the Natural Science Foundation of Shandong Province (ZR2021ZD14, ZR2021JQ24), and Science and Technology Project of Qingdao West Coast New Area (2019-32, 2020-20, 2020-1-4), High-level Talent Team Project of Qingdao West Coast New Area (RCTD-JC-2019-05), Key Research and Development Program of Shandong Province (2020CXGC01208).
 *Corresponding author.}
\thanks{Y. Sun is with the School of Control Science and Engineering, Shandong University, Jinan, Shandong 250061, China (e-mail: sunyue9603@163.com). H. Li and H. Zhang are with College of Electrical Engineering and Automation, Shandong University of Science and Technology, Qingdao, Shandong 266590, China (e-mail: hszhang@sdu.edu.cn; lhd200908@163.com).}
}



\maketitle

\begin{abstract}
In this paper, the two-player leader-follower game with private inputs for feedback Stackelberg strategy is considered.
In particular, the follower shares its measurement information with the leader except its historical control inputs
while the leader shares none of the historical control inputs and the measurement information with the follower.
The private inputs of the leader and the follower lead to the main obstacle,
which causes the fact that the estimation gain and the control gain are related with each other,
resulting that the forward and backward Riccati equations are coupled and making the calculation complicated.
By introducing a kind of novel observers through the information structure for the follower and the leader, respectively,
a kind of new observer-feedback Stacklberg strategy is designed.
Accordingly, the above-mentioned obstacle is also avoided.
Moreover, it is found that the cost functions under the presented observer-feedback Stackelberg strategy are asymptotically optimal to the cost functions under the optimal feedback Stackelberg strategy with the feedback form of the state.
Finally, a numerical example is given to show the efficiency of this paper.
\end{abstract}



\section{Introduction}
In the traditional control model, centralized control is a basic concept and has been extensively studied from time-invariant system to time-variant system and system with time-delay \cite{Anderson1971, Rami2001, HSZhang2015}.
However, with the development of wireless sensor network and artificial intelligence, the centralized control will no longer be applicable due to the fact that the achievable bandwidth would be limited by long delays induced by the communication between the centralized controller \cite{Bauer2013}.
The task of effectively controlling multiple decision-makers systems in the absence of communication channels is increasingly an interesting and challenging control problem.
Correspondingly, the decentralized control of large-scale systems arises accordingly,
which has widespread implementation in electrical power distribution networks, cloud environments, multi-agent systems, reinforcement learning and so on \cite{Blaabjerg2006, Hoogenkamp2022, Ha2004, Gorges2019},
where decisions are made by multiple different decision-makers who have access to different information.

Decentralized control can be traced back to 1970s \cite{Witsenhausen1968, Davison1973, Davison1976}.
The optimization of decentralized control can be divided into two categories.
The first category is the decentralized control for multi-controllers with one associated cost function \cite{Yoshikawa1975, Swigart2010, Liang2023}.
Nayyar studied decentralized stochastic control with partial history observations and control inputs sharing  in \cite{Nayyar2013} by using the common information approach and the $n$-step delayed sharing information structure was investigated in \cite{Nayyar2011}.
\cite{Liang2021} focused on decentralized control in networked control system with asymmetric information by solving the forward and backward coupled Riccati equations through forward iteration, where the historical control inputs was shared unilaterally compared with the information structure shared with each other in \cite{Nayyar2013, Nayyar2011}.
\cite{BWang2019} designed decentralized strategies for mean-field system, which was further shown to have asymptotic robust social optimality.
The other category is the decentralized control for game theory \cite{Pachter2017, Sun2022, Zhi2018}.
Two-criteria LQG decision problems with one-step delay observation sharing pattern for stochastic discrete-time system in Stackelberg strategy and Nash equilibrium strategy were considered in \cite{Basar1978} and \cite{George1982}, respectively.
Necessary conditions for an optimal Stackelberg strategy with output feedback form were given in \cite{Suzumura1993} with incomplete information of the controllers.
\cite{Klompstra2000} investigated feedback risk-sensitive Nash equilibrium solutions for two-player nonzero-sum games with complete state observation and shared historical control inputs.
Static output feedback incentive Stackelberg game with markov jump for linear stochastic systems was taken into consideration in \cite{Mukaidani2016}
and a numerical algorithm was further proposed which guaranteed local convergence.

Noting that the information structure in the decentralized control systems mentioned above has the following feature, that is,
all or part of historical control inputs of the controllers are shared with the other controllers.
However, the case, where the controllers have its own private control inputs,
has not been addressed in decentralized control system, which has applications in  a personalized healthcare setting, in the states of a virtual keyboard user (e.g., Google GBoard users) and in the social robot for second language education of children \cite{Chowdhury2021}.
It should be noted that the information structure where the control information are unavailable to the other decision makers will
 cause the estimation gain depends on the control gain and vice versa,
which means the forward and backward Riccati equations are coupled, and
make the calculation more complicated.
Motivated by \cite{Juan2023}, which focused on the LQ optimal control problem of linear systems with private input and measurement information
by using a kind of novel observers to overcome the obstacle,
in this paper, we are concerned with the feedback Stackelberg strategy for two-player game with private control inputs.
In particular, the follower shares its measurement information to the leader, while the leader doesn't share any information to the follower due to the hierarchical relationship and the historical control inputs for the follower and the leader are both private,
which is the main obstacle in this paper. To overcome the problem,
firstly, the novel observers based on the information structure of each controller are proposed.
Accordingly, a new kind of observer-feedback Stackelberg strategy for the follower and the leader is designed.
Finally, it proved that the associated cost functions for the follower and the leader under the proposed observer-feedback Stackelberg strategy
are asymptotically optimal as compared with the cost functions under the optimal feedback Stackelberg strategy with the feedback form of the state obtained in \cite{Castanon1976}.

The outline of this paper is given as follows.
The problem formulation is given in Section II.
The observers and the observer-feedback Stackelberg strategy with private inputs are designed in Section III.
The asymptotical optimal analysis is shown in Section IV.
Numerical examples are presented in Section V.
Conclusion is given in Section VI.

\emph{Notations}: $\mathbb{R}^n$ represents the space of all real $n$-dimensional vectors.
$A'$ means the transpose of the matrix $A$. A symmetric matrix $A>0$ (or $A\geq 0$) represents that the matrix $A$ is positive definite (or positive semi-definite).
$\|x\|$ denotes the Euclidean norm of vector $x$, i.e., $\|x\|^2=x'x$.
$\|A\|$ denotes the Euclidean norm of matrix $A$, i.e., $\|A\|=\sqrt{\lambda_{max}(A'A)}$.
$\lambda(A)$ represents the eigenvalues of the matrix $A$ and
$\lambda_{max}(A)$ represents the largest eigenvalues of the matrix $A$.
$I$ is an identity matrix with compatible dimension.
$0$ in block matrix represents a zero matrix with appropriate dimensions.

\section{Problem Formulation}
Consider a two-player leader-follower game described as:
\begin{eqnarray}
  x(k+1) &\hspace{-0.8em}=&\hspace{-0.8em} Ax(k)+B_1u_1(k)+B_2u_2(k),\label{1-1}\\
  y_1(k)&\hspace{-0.8em}=&\hspace{-0.8em}H_1x(k),\label{1-2}\\
  y_2(k)&\hspace{-0.8em}=&\hspace{-0.8em}H_2x(k),\label{1-3}
\end{eqnarray}
where $x(k)\in \mathbb{R}^n$ is the state with initial value $x(0)$. $u_1(k)\in \mathbb{R}^{m_1}$ and $u_2(k)\in \mathbb{R}^{m_2}$ are the two control inputs of the follower and the leader, respectively.
$y_i(k)\in \mathbb{R}^{s_i}$ is the measurement information.
$A$, $B_i$ and $H_i$ ($i=1, 2$) are constant matrices with compatible dimensions.

The associated cost functions for the follower and the leader are given by
\begin{eqnarray}
  J_1&\hspace{-0.8em}=&\hspace{-0.8em} \sum\limits^{\infty}_{k=0}[x'(k)Q_1x(k)+u'_1(k)R_{11}u_1(k)\nonumber\\
  &\hspace{-0.8em} &\hspace{-0.8em} +u'_2(k)R_{12}u_2(k)],\label{2-1}\\
  J_2&\hspace{-0.8em}=&\hspace{-0.8em} \sum\limits^{\infty}_{k=0}[x'(k)Q_2x(k)+u'_1(k)R_{21}u_1(k)\nonumber\\
  &\hspace{-0.8em} &\hspace{-0.8em}+u'_2(k)R_{22}u_2(k)],\label{2-2}
\end{eqnarray}
where the weight matrices are such that $Q_i\geq0 $, $R_{ij}\geq0$ ($i\neq j$) and $R_{ii}>0$ ($i, j= 1, 2$) with compatible dimensions.

Feedback Stackelberg strategy with different information structure for controllers had been considered since 1970s in \cite{Castanon1976}, where the information structure satisfied that the controller shared all or part of historical inputs to the other.
To the best of our knowledge, there has been no efficiency technique to deal with the case of private inputs for controllers.
The difficultly lies in the unavailability of other controllers' historical control inputs,
which leads to the fact that the estimation gain depends on the control gain
and makes the forward and backward Riccati equations coupled.
In this paper, our goal is that by designing the novel observers based on the measurements and private inputs for the follower and the leader, respectively,
we will show the proposed observer-feedback Stackelberg strategy is asymptotic optimal to the deterministic case in  \cite{Castanon1976}.
Mathematically, by denoting
\begin{eqnarray}
  Y_i(k)&\hspace{-0.8em}=&\hspace{-0.8em} \{y_i(0), ..., y_i(k)\}, \nonumber\\
  U_i(k-1)&\hspace{-0.8em}=&\hspace{-0.8em} \{u_i(0), ..., u_i(k-1)\},\nonumber\\
  {F}_1(k)&\hspace{-0.8em}=&\hspace{-0.8em}\{Y_1(k), U_1(k-1)\}, \label{sy-1}\\
   {F}_2(k)&\hspace{-0.8em}=&\hspace{-0.8em}\{Y_1(k), Y_2(k), U_2(k-1)\},\label{sy-2}
\end{eqnarray}
we will design the observer-feedback Stackelberg strategy  based on the information $\mathcal{F}_i(k)$, where $u_i(k)$ is $\mathcal{F}_i(k)$-casual for
$i= 1, 2$ in this paper.
The following assumptions will be used in this paper. \smallskip
\begin{assumption}\label{assum1}
System $(A, B)$ is stabilizable with $B=\left[\hspace{-0.4em}
                                          \begin{array}{cc}
                                            B_1 & B_2 \\
                                          \end{array}
                                        \hspace{-0.4em}\right]
$ and system $(A, Q_i)$ ($i=1, 2$) is observable.
\end{assumption}

By denoting the admissible controls sets $\mathcal{U}_i$ (i=1, 2) for the feedback Stackelberg strategy of the follower and the leader:
\begin{eqnarray}\label{3}
\mathcal{U}_1&\hspace{-0.8em}=\{&\hspace{-0.8em}u_1: \Omega\times[0, N]\times \mathbb{R}^n \times U_2 \longrightarrow U_1 \},\nonumber\\
  \mathcal{U}_2&\hspace{-0.8em}=\{&\hspace{-0.8em} u_2: \Omega\times[0, N]\times \mathbb{R}^n \longrightarrow U_2\},
\end{eqnarray}
where $U_1$ and $U_2$ represent the strategy for the follower and the leader, respectively, the definition of the feedback Stackelberg strategy
\cite{Bensoussan2015} is given.
\begin{definition}\label{def1}
$(u^*_1(k), u^*_2(k))\in \mathcal{U}_1\times \mathcal{U}_2$ is the optimal feedback Stackelberg strategy, if there holds that:
\begin{eqnarray}\label{4}
J_1(u^{*}_1(k, u^*_2(k)), u^*_2(k))&\hspace{-0.8em}\leq&\hspace{-0.8em} J_1(u_1(k, u^*_2(k)), u^*_2(k)), \forall u_1\in \mathcal{U}_1,\nonumber\\
J_2(u^{*}_1(k, u^*_2(k)), u^*_2(k))&\hspace{-0.8em}\leq &\hspace{-0.8em}J_2(u^{*}_1(k, u_2(k)), u_2(k)), \forall u_2\in \mathcal{U}_2.\nonumber
\end{eqnarray}
\end{definition}

Firstly, the optimal feedback Stackelberg strategy in deterministic case with perfect information structure is given,
that is, the information structure of the follower and the leader both satisfy
\begin{eqnarray*}
Y_k=\{x(0), ..., x(k), u_i(0), ..., u_i(k-1), \quad i=1,2\}.
\end{eqnarray*}
\begin{lemma}\label{lem1}
Under Assumption \ref{assum1}, the optimal feedback Stackelberg strategy with the information structure for the follower and the leader satisfying $Y_k$,
is given by
\begin{eqnarray}
  u_1(k) &\hspace{-0.8em}=&\hspace{-0.8em} K_1x(k),\label{5-1}\\
  u_2(k) &\hspace{-0.8em}=&\hspace{-0.8em} K_2x(k), \label{5-2}
\end{eqnarray}
where the feedback gain matrices $K_1$ and $K_2$ satisfy
\begin{eqnarray}
K_1 &\hspace{-0.8em}=&\hspace{-0.8em} -\Gamma^{-1}_{1}Y_{1},\label{6-1}\\
K_2 &\hspace{-0.8em}=&\hspace{-0.8em} -\Gamma^{-1}_{2}Y_{2},\label{6-2}
\end{eqnarray}
with
\begin{eqnarray*}
\Gamma_{1} &\hspace{-0.8em}=&\hspace{-0.8em} R_{11}+B'_1P_{1}B_1,\label{7-1} \\
\Gamma_{2} &\hspace{-0.8em}=&\hspace{-0.8em}R_{22}+B'_2M'_{1}P_{2}M_{1}B_2+B'_2S'R_{21}SB_2,\label{7-2}\\
M_{1} &\hspace{-0.8em}=&\hspace{-0.8em}I-B_1S,\label{7-3} \quad S=\Gamma^{-1}_{1}B'_1P_{1},\label{7-4}\\
Y_{1} &\hspace{-0.8em}=&\hspace{-0.8em} B'_1P_{1}A+B'_1P_{1}B_2K_2,\label{7-5}\\
Y_{2} &\hspace{-0.8em}=&\hspace{-0.8em}B'_2M'_{1}P_{2}M_{1}A+B'_2S'R_{21}SA,\label{7-6}
\end{eqnarray*}
where $P_1$ and $P_2$ satisfy the following two-coupled algebraic Riccati equations:
\begin{eqnarray}
P_1&\hspace{-0.8em}=&\hspace{-0.8em}Q_1+(A+B_2K_2)'P_{1}(A+B_2K_2)\nonumber\\
&\hspace{-0.8em}&\hspace{-0.8em}-Y'_{1}\Gamma^{-1}_{1}Y_{1}+K'_2R_{12}K_2, \label{8-1} \\
  P_2&\hspace{-0.8em}=&\hspace{-0.8em}Q_2+A'M'_{1}P_{2}M_{1}A+A'S'R_{21}SA\nonumber\\
 &\hspace{-0.8em}&\hspace{-0.8em} -Y'_{2}\Gamma^{-1}_{2}Y_{2}. \label{8-2}
\end{eqnarray}

The optimal cost functions for feedback Stackelberg strategy are such that
\begin{eqnarray}
  J^*_1 &\hspace{-0.8em}=&\hspace{-0.8em}x'(0)P_1x(0),\label{9-1}\\
  J^*_2 &\hspace{-0.8em}=&\hspace{-0.8em}x'(0)P_2x(0).\label{9-2}
\end{eqnarray}
\end{lemma}

\begin{proof}
The optimal feedback Stackelberg strategy for deterministic case with perfect information structure for the follower and the leader in finite-time horizon has been shown in (18)-(28) with $\theta(t)=\Pi_1(t)=\Pi_2(t)=0$ in \cite{Castanon1976}. By using the results in Theorem 2 in \cite{HSZhang2015}, the results obtained in \cite{Castanon1976} can be extended into infinite horizon,
i.e., (18)-(28) in \cite{Castanon1976} are convergent to the algebraic equations obtained in (\ref{6-1})-(\ref{6-2}) and (\ref{8-1})-(\ref{8-2}) in Lemma \ref{lem1} of this paper
by using the monotonic boundedness theorem.  This completes the proof.
\end{proof}
\smallskip
\begin{remark}
$P_1>0$ and $P_2>0$ in (\ref{8-1})-(\ref{8-2}) can be shown accordingly by using Theorem 2 in \cite{HSZhang2015}, which guaranteed the invertibility of $\Gamma_1$ and $\Gamma_2$.
\end{remark}

\begin{remark}
Compared with \cite{Castanon1976}, where the historical control inputs of the follower and the leader are shared with each other,
the historical control inputs of this paper are private, leading to the main obstacle.
\end{remark}

\section{The observer-feedback Stackelberg strategy}
Based on the discussion above, we are in position to consider the leader-follower game with private inputs, i.e., ${u}_i(k)$ is ${F}_i(k)$-casual.
\begin{remark}
As pointed out in \cite{Liang2021}, the information structure in decentralized control, where one of the controllers (C1) doesn't share the historical control inputs to the other controller (C2) while C2 shares its historical control inputs with C1, is a challenge problem due to the control gain and estimator gain are coupled. The difficulty with private inputs for the follower and the leader is even more complicated due to the unavailability of the historical control inputs of each controller.
\end{remark}

Considering the private inputs of the follower and the leader, the observers $\hat{x}_i(k)$ ($i=1, 2$) are designed as follows:
\begin{eqnarray}
\hat{x}_1(k+1)&\hspace{-0.8em}=&\hspace{-0.8em}A\hat{x}_1(k)+B_1u^{\star}_1(k)+B_2K_2\hat{x}_1(k)\nonumber\\
&\hspace{-0.8em}&\hspace{-0.8em}+L_1[y_1(k)-H_1\hat{x}_1(k)],\label{11-1}\\
\hat{x}_2(k+1)&\hspace{-0.8em}=&\hspace{-0.8em}A\hat{x}_2(k)+B_1K_1\hat{x}_2(k)+B_2u^{\star}_2(k)\nonumber\\
&\hspace{-0.8em}&\hspace{-0.8em}+L_2[y_2(k)-H_2\hat{x}_2(k)],\label{11-2}
\end{eqnarray}
where the observer gain matrices $L_1$ and $L_2$ are chosen to make the observers stable.
Accordingly, the observer-feedback Stackelberg strategy is designed as follows:
\begin{eqnarray}
 u^{\star}_1(k) &\hspace{-0.8em}=&\hspace{-0.8em} K_1\hat{x}_1(k),\label{10-1}\\
  u^{\star}_2(k) &\hspace{-0.8em}=&\hspace{-0.8em} K_2\hat{x}_2(k), \label{10-2}
\end{eqnarray}
where $K_1$ and $K_2$ are given in (\ref{6-1})-(\ref{6-2}), respectively.

For convenience of future discussion, some symbols will be given beforehand.
\begin{eqnarray}
  \mathcal{A} &\hspace{-0.8em}=&\hspace{-0.8em} \left[\hspace{-0.3em}
                                                  \begin{array}{cc}
                                                    A+B_2K_2-L_1H_1 &\hspace{-1em} -B_2K_2 \\
                                                    -B_1K_1 &\hspace{-1em} A+B_1K_1-L_2H_2 \\
                                                  \end{array}
                                                \hspace{-0.3em}\right],\label{sy-11}\\
  \mathcal{B} &\hspace{-0.8em}=&\hspace{-0.8em}\left[\hspace{-0.3em}
                                                 \begin{array}{cc}
                                                   -B_1K_1 & -B_2K_2 \\
                                                 \end{array}
                                               \hspace{-0.3em}\right]\nonumber\\
              &\hspace{-0.8em}=&\hspace{-0.8em}\left[\hspace{-0.3em}
                                                 \begin{array}{cc}
                                                   B_1S(A+B_2K_2) & -B_2K_2 \\
                                                 \end{array}
                                               \hspace{-0.3em}\right],\nonumber\\
  \bar{A}&\hspace{-0.8em}=&\hspace{-0.8em}\left[\hspace{-0.3em}
                                            \begin{array}{cc}
                                              A+B_1K_1+B_2K_2 & \mathcal{B} \\
                                              0 & \mathcal{A} \\
                                            \end{array}
                                          \hspace{-0.3em}\right],\nonumber\\
  \tilde{x}(k)&\hspace{-0.8em}=&\hspace{-0.8em}\left[\hspace{-0.3em}
                 \begin{array}{cc}
                  \tilde{x}'_1(k) & \tilde{x}'_2(k) \\
                 \end{array}
               \hspace{-0.3em}\right]',\nonumber\\
  \tilde{x}_i(k)&\hspace{-0.8em}=&\hspace{-0.8em} x(k)-\hat{x}_i(k), \quad i=1, 2.\nonumber
\end{eqnarray}
Subsequently, the stability of the observers $\hat{x}_i(k)$ ($i=1, 2$) and the stability of the closed-loop system (\ref{1-1}) under the designed
observer-feedback Stackelberg strategy (\ref{10-1})-(\ref{10-2}) are shown, respectively.

\begin{theorem}\label{thm1}
If there exist optional gain matrices $L_1$ and $L_2$ such that the matrix $\mathcal{A}$ is stable,
then, the observers $\hat{x}_i(k)$ for $i=1, 2$ are stable with the controllers of the follower and the leader satisfying (\ref{10-1})-(\ref{10-2}), i.e., there holds
\begin{eqnarray}\label{sy-1}
  \lim_{k\rightarrow \infty}\|x(k)-\hat{x}_i(k)\|=0.
\end{eqnarray}
\end{theorem}

\begin{proof}
By substituting the observer-feedback controllers (\ref{10-1})-(\ref{10-2}) into (\ref{1-1}), then $x(k+1)$ is recalculated as:
\begin{eqnarray}\label{sy-2}
x(k+1)&\hspace{-0.8em}=&\hspace{-0.8em} Ax(k)+B_1K_1\hat{x}_1(k)+B_2K_2\hat{x}_(k)\nonumber\\
&\hspace{-0.8em}=&\hspace{-0.8em}[A+B_1K_1+B_2K_2]x(k)-B_1K_1\tilde{x}_1(k)\nonumber\\
&\hspace{-0.8em}&\hspace{-0.8em}-B_2K_2\tilde{x}_2(k).
\end{eqnarray}
Accordingly, by adding (\ref{10-1})-(\ref{10-2}) into the observers (\ref{11-1})-(\ref{11-2}) and combining with (\ref{sy-2}), the derivation of $\tilde{x}_i(k)$ for $i=1, 2$ are given as
\begin{eqnarray*}
\tilde{x}_1(k+1)&\hspace{-0.8em}=&\hspace{-0.8em}(A+B_2K_2-L_1H_1)\tilde{x}_1(k)-B_2K_2\tilde{x}_2(k),\\
\tilde{x}_2(k+1)&\hspace{-0.8em}=&\hspace{-0.8em}(A+B_1K_1-L_2H_2)\tilde{x}_1(k)-B_1K_1\tilde{x}_1(k),
\end{eqnarray*}
that is
\begin{eqnarray}\label{12-1}
 \tilde{x}(k+1)&\hspace{-0.8em}=&\hspace{-0.8em}\mathcal{A}\tilde{x}(k).
\end{eqnarray}
Subsequently, if there exist matrices $L_1$ and $L_2$ making $\mathcal{A}$ stable, then, the stability of the matrix $\mathcal{A}$ means that
\begin{eqnarray*}
  \lim_{k\rightarrow \infty} \tilde{x}(k)=0,
\end{eqnarray*}
i.e., (\ref{sy-1}) is established. That is to say, the observers $\hat{x}_i(k)$ are stable under (\ref{10-1})-(\ref{10-2}).
The proof is completed.
\end{proof}

\begin{remark}
Noting that in Theorem \ref{thm1} the key point lies in that how to select $L_i$ ($i=1, 2$) so that the eigenvalues of
the matrix $\mathcal{A}$ are within the unit circle.
The following analysis gives an  method to find $L_i$.

According to the Lyapunov stability criterion, i.e., $\mathcal{A}$ is stable if and only if for any positive definite matrix $Q$,
$\mathcal{A}'P\mathcal{A}-P =-Q$ admits a solution such that $P>0$.
Thus, if there exists a $P>0$ such that
\begin{eqnarray}
\mathcal{A}'P\mathcal{A}-P<0,
\end{eqnarray}
then $\mathcal{A}$ is stable. Following from the elementary row transformation, one has
\begin{eqnarray*}
&\hspace{-0.8em}&\hspace{-0.8em}\left(
  \begin{array}{cc}
    I & I \\
    0 & I \\
  \end{array}
\right)\left(
  \begin{array}{cc}
    I & 0 \\
    0 & \mathcal{A}' \\
  \end{array}
\right)\left(
  \begin{array}{cc}
    -P & \mathcal{A}'P \\
    P\mathcal{A} & -P \\
  \end{array}
\right)\left(
  \begin{array}{cc}
    I & 0 \\
    0 & \mathcal{A} \\
  \end{array}
\right)\nonumber\\
&\hspace{-0.8em}&\hspace{-0.8em}\times \left(
  \begin{array}{cc}
    I & 0 \\
    I & I \\
  \end{array}
\right)=\left(
  \begin{array}{cc}
    \mathcal{A}'P\mathcal{A}-P & 0 \\
    0 & -\mathcal{A}'P\mathcal{A} \\
  \end{array}
\right)<0,
\end{eqnarray*}
that is, $\mathcal{A}'P\mathcal{A}-P<0$ is equivalent to the following matrix inequality
\begin{eqnarray}\label{0908-1}
\left(
  \begin{array}{cc}
    -P & \mathcal{A}'P \\
    P\mathcal{A} & -P \\
  \end{array}
\right)<0.
\end{eqnarray}
Noting that $\mathcal{A}$ is related with $L_i$,
in order to use the linear matrix inequality (LMI) Toolbox in Matlab to find $L_i$, (\ref{0908-1}) will be transmit into a LMI form.
Let
\begin{eqnarray*}
P=\left(
  \begin{array}{cc}
    P & 0 \\
    0 & P \\
  \end{array}
\right), \quad \tilde{W}=\left(
  \begin{array}{cc}
    W_1 & 0 \\
    0 & W_2 \\
  \end{array}
\right),
\end{eqnarray*}
and rewrite $\mathcal{A}$ in (\ref{sy-11}) as $\mathcal{A}=\tilde{A}-\tilde{L}\tilde{H}$, where
\begin{eqnarray*}
\mathcal{A} &\hspace{-0.8em}=&\hspace{-0.8em}\left(
                                                  \begin{array}{cc}
                                                    A+B_2K_2 &\hspace{-1em} -B_2K_2 \\
                                                    -B_1K_1 &\hspace{-1em} A+B_1K_1 \\
                                                  \end{array}
                                                \right),\\
\tilde{L}&\hspace{-0.8em}=&\hspace{-0.8em}\left(
  \begin{array}{cc}
    L_1 & 0 \\
    0 & L_2 \\
  \end{array}
\right),\quad  \tilde{H}=\left(
  \begin{array}{cc}
    H_1 & 0 \\
    0 & H_2 \\
  \end{array}
\right).
\end{eqnarray*}
To this end, we have
\begin{eqnarray*}
  P\mathcal{A}=P\tilde{A}-P\tilde{L}\tilde{H}=P\tilde{A}-\tilde{W}\tilde{H},
\end{eqnarray*}
with $\tilde{W}=P\tilde{L}$.
Based on the discussion above, it concludes that $\mathcal{A}$ is stable if there exists a $P>0$ such that the following LMI:
\begin{eqnarray}\label{0908-2}
\left(
  \begin{array}{cc}
    -P & (P\tilde{A}-\tilde{W}\tilde{H})'\\
    P\tilde{A}-\tilde{W}\tilde{H} & -P \\
  \end{array}
\right)<0.
\end{eqnarray}

In this way, by using the LMI Toolbox in Matlab, $L_i$ can be found according, which stabilizes $\mathcal{A}$ where $L_i=P^{-1}W_i$.

\end{remark}

Under the observer-feedback controllers (\ref{10-1})-(\ref{10-2}), the stability of (\ref{1-1}) is given.
\begin{theorem}\label{thm2}
Under Assumption \ref{assum1} and if there exists $L_i$ stabilizing $\mathcal{A}$, then the closed-loop system (\ref{1-1}) is stable with the observer-feedback controllers
(\ref{10-1})-(\ref{10-2}).
\end{theorem}
\begin{proof}
According to (\ref{sy-2}), the closed-loop system (\ref{1-1}) is reformulated as
\begin{eqnarray}\label{sy-3}
x(k+1)&\hspace{-0.8em}=&\hspace{-0.8em} [A+B_1K_1+B_2K_2]x(k)+\mathcal{B}\tilde{x}(k).
\end{eqnarray}
Together with (\ref{12-1}), we have
\begin{eqnarray}\label{sy-4}
\left[
  \begin{array}{c}
    x(k+1) \\
    \tilde{x}(k+1) \\
  \end{array}
\right]&\hspace{-0.8em}=&\hspace{-0.8em}\bar{A}\left[
  \begin{array}{c}
    x(k) \\
    \tilde{x}(k) \\
  \end{array}
\right].
\end{eqnarray}
The stability of $A+B_1K_1+B_2K_2$ is guaranteed by the stabilizability of $(A, B)$ and the observability of $(A, Q_i)$ for $i=1, 2$. Following from Theorem \ref{thm1}, $\mathcal{A}$ is stabilized by selecting appropriate gain matrices $L_1$ and $L_2$. Subsequently, the stability of the closed-loop system (\ref{1-1}) is derived.
This completes the proof.
\end{proof}

\section{The asymptotical optimal analysis}
The stability of the state and the observers, i.e., $x(k)$ and $\hat{x}_i$ for $i=1, 2$ has been shown in Theorem \ref{thm1} and Theorem \ref{thm2} under the observer-feedback controllers (\ref{10-1})-(\ref{10-2}).
To shown the rationality of the design of the observer-feedback controllers (\ref{10-1})-(\ref{10-2}),
the asymptotical optimal analysis relating with the cost functions under (\ref{10-1})-(\ref{10-2}) is given.
To this end, denote the cost functions for the follower and the leader satisfying
\begin{eqnarray}
  J_1(s, M)&\hspace{-0.8em}=&\hspace{-0.8em} \sum\limits^{M}_{k=s}[x'(k)Q_1x(k)+u'_1(k)R_{11}u_1(k)\nonumber\\
  &\hspace{-0.8em}&\hspace{-0.8em}+u'_2(k)R_{12}u_2(k)],\label{13-1}\\
  J_2(s, M)&\hspace{-0.8em}=&\hspace{-0.8em} \sum\limits^{M}_{k=s}[x'(k)Q_2x(k)+u'_1(k)R_{21}u_1(k)\nonumber\\
  &\hspace{-0.8em}&\hspace{-0.8em}+u'_2(k)R_{22}u_2(k)].\label{13-2}
\end{eqnarray}

Now, we are in position to show that the observer-feedback Stackelberg strategy (\ref{10-1})-(\ref{10-2}) is asymptotical optimal to the optimal feedback Stackelberg strategy presented in Lemma \ref{lem1}.
\begin{theorem}\label{thm3}
Under Assumption \ref{assum1}, the corresponding cost functions (\ref{13-1})-(\ref{13-2}) under the observer-feedback Stackelberg strategy (\ref{10-1})-(\ref{10-2}) with $L_i$ ($i=1, 2$) selected from Theorem \ref{thm1} are given by
\begin{eqnarray}
  J^{\star}_1(s, \infty) &\hspace{-0.8em}=&\hspace{-0.8em} x'(s)P_1x(s)\nonumber\\
   &\hspace{-0.8em} &\hspace{-0.8em}+\sum\limits^{\infty}_{k=s} \left[
  \begin{array}{c}
    x(k) \\
    \tilde{x}(k) \\
  \end{array}
\right]'\left[
          \begin{array}{cc}
            0 & T_1 \\
            T'_1 & S_1\\
          \end{array}
        \right]
\left[
  \begin{array}{c}
    x(k) \\
    \tilde{x}(k) \\
  \end{array}
\right],\label{14-1}\\
  J^{\star}_2(s, \infty) &\hspace{-0.8em}=&\hspace{-0.8em} x'(s)P_2x(s)\nonumber\\
   &\hspace{-0.8em} &\hspace{-0.8em}+\sum\limits^{\infty}_{k=s} \left[
  \begin{array}{c}
    x(k) \\
    \tilde{x}(k) \\
  \end{array}
\right]'\left[
          \begin{array}{cc}
            0 & T_2\\
            T'_2 & S_2\\
          \end{array}
        \right]
\left[
  \begin{array}{c}
    x(k) \\
    \tilde{x}(k) \\
  \end{array}
\right],\label{14-2}
\end{eqnarray}
where
\begin{eqnarray*}
S_1 &\hspace{-0.8em}=&\hspace{-0.8em}  \mathcal{B}'P_1 \mathcal{B}-\left[
                                         \begin{array}{cc}
                                           K'_1R_{11}K_1 & 0 \\
                                           0 & K'_2R_{12}K_2 \\
                                         \end{array}
                                       \right],\\
S_2 &\hspace{-0.8em}=&\hspace{-0.8em}  \mathcal{B}'P_2 \mathcal{B}-\left[
                                         \begin{array}{cc}
                                           K'_1R_{21}K_1 & 0 \\
                                           0 & K'_2R_{22}K_2 \\
                                         \end{array}
                                       \right],\\
T_1&\hspace{-0.8em}=&\hspace{-0.8em}(A+B_2K_2)'M'_1P_1 \mathcal{B},\\
T_2&\hspace{-0.8em}=&\hspace{-0.8em}(A+B_2K_2)'M'_1P_2 \mathcal{B}.
\end{eqnarray*}
Moreover, the differences, which are denoted as $\delta J_1(s, \infty)$ and $\delta J_2(s, \infty)$, between (\ref{14-1})-(\ref{14-2}) and the optimal cost functions (\ref{9-1})-(\ref{9-2}) obtained in Lemma \ref{lem1} under the optimal feedback Stackelberg strategy  are such that
\begin{eqnarray}
  \delta J_1(s, \infty) &\hspace{-0.8em}=&\hspace{-0.8em} J^{\star}_1(s, \infty)-J^*_1(s, \infty)\nonumber\\
  &\hspace{-0.8em}=&\hspace{-0.8em}\sum\limits^{\infty}_{k=s} \left[
  \begin{array}{c}
    x(k) \\
    \tilde{x}(k) \\
  \end{array}
\right]'\left[
          \begin{array}{cc}
            0 & T_1\\
            T'_1 & S_1\\
          \end{array}
        \right]
\left[
  \begin{array}{c}
    x(k) \\
    \tilde{x}(k) \\
  \end{array}
\right],\label{15-1}\\
\delta J_2(s, \infty) &\hspace{-0.8em}=&\hspace{-0.8em} J^{\star}_2(s, \infty)-J^*_2(s, \infty)\nonumber\\
  &\hspace{-0.8em}=&\hspace{-0.8em}\sum\limits^{\infty}_{k=s} \left[
  \begin{array}{c}
    x(k) \\
    \tilde{x}(k) \\
  \end{array}
\right]'\left[
          \begin{array}{cc}
            0 & T_2\\
            T'_2 & S_2\\
          \end{array}
        \right]
\left[
  \begin{array}{c}
    x(k) \\
    \tilde{x}(k) \\
  \end{array}
\right].\label{15-2}
\end{eqnarray}
\end{theorem}

\begin{proof}
The proof will be divided into two parts. The first part is to consider the cost function of the follower under the observer-feedback controllers (\ref{10-1})-(\ref{10-2}).
Following from (\ref{sy-2}), system (\ref{1-1}) it can be rewritten as
\begin{eqnarray}\label{16}
x(k+1)&\hspace{-0.8em}=&\hspace{-0.8em} [A+B_1K_1+B_2K_2]x(k)-B_1K_1\tilde{x}_1(k)\nonumber\\
&\hspace{-0.8em}&\hspace{-0.8em}-B_2K_2\tilde{x}_2(k)\nonumber\\
&\hspace{-0.8em}=&\hspace{-0.8em}(I-B_1S)(A+B_2K_2)x(k)+\mathcal{B}\tilde{x}(k),
\end{eqnarray}
where $K_1$ in (\ref{6-1}) have been used in the derivation of the last equality.

Firstly, we will prove $J^{\star}_1(s, \infty)$ satisfies (\ref{14-1}). Combing (\ref{16}) with (\ref{8-1}), one has
\begin{eqnarray}\label{17}
&\hspace{-0.8em}&\hspace{-0.8em} x'(k)P_1x(k)-x(k+1)'P_1x(k+1)\nonumber\\
=&\hspace{-0.8em}&\hspace{-0.8em}x'(k)[P_1-(A+B_2K_2)'(I-B_1S)'P_1(I-B_1S)\nonumber\\
&\hspace{-0.8em}&\hspace{-0.8em}\times(A+B_2K_2)]x(k)-x'(k)(A+B_2K_2)'M'_1P_1\mathcal{B}\tilde{x}(k)\nonumber\\
&\hspace{-0.8em}&\hspace{-0.8em}-\tilde{x}'(k)\mathcal{B}'P_1M_1(A+B_2K_2)x(k)
-\tilde{x}'(k)\mathcal{B}'P_1\mathcal{B}\tilde{x}(k)\nonumber\\
=&\hspace{-0.8em}&\hspace{-0.8em}x'(k)[Q_1+K'_2R_{12}K_2-(A+B_2K_2)'P_1B_1\Gamma^{-1}_{1}B'_{1}P_1\nonumber\\
&\hspace{-0.8em}&\hspace{-0.8em}\times (A+B_2K_2)+(A+B_2K_2)'P_1B_1S(A+B_2K_2)\nonumber\\
&\hspace{-0.8em}&\hspace{-0.8em}+(A+B_2K_2)'S'B'_{1}P_1(A+B_2K_2)-(A+B_2K_2)'S'\nonumber\\
&\hspace{-0.8em}&\hspace{-0.8em}\times B'_{1}P_1B_1S(A+B_2K_2)]x(k)-x'(k)(A+B_2K_2)'M'_1\nonumber\\
&\hspace{-0.8em}&\hspace{-0.8em}\times P_1\mathcal{B}\tilde{x}(k)-\tilde{x}'(k)\mathcal{B}'P_1M_1(A+B_2K_2)x(k)\nonumber\\
&\hspace{-0.8em}&\hspace{-0.8em}-\tilde{x}'(k)\mathcal{B}'P_1\mathcal{B}\tilde{x}(k)\nonumber\\
=&\hspace{-0.8em}&\hspace{-0.8em}x'(k)[Q_1+K'_2R_{12}K_2+K'_1(R_{11}+B'_{1}P_1B_1)K_1\nonumber\\
&\hspace{-0.8em}&\hspace{-0.8em}-K'_1B'_{1}P_1B_1K_1]x(k)-x'(k)(A+B_2K_2)'M'_1P_1\mathcal{B}\tilde{x}(k)\nonumber\\
&\hspace{-0.8em}&\hspace{-0.8em}-\tilde{x}'(k)\mathcal{B}'P_1M_1(A+B_2K_2)x(k)
-\tilde{x}'(k)\mathcal{B}'P_1\mathcal{B}\tilde{x}(k)\nonumber\\
=&\hspace{-0.8em}&\hspace{-0.8em}x'(k)[Q_1+K'_1R_{11}K_1+K'_2R_{12}K_2]x(k)\nonumber\\
&\hspace{-0.8em}&\hspace{-0.8em}-x'(k)(A+B_2K_2)'M'_1P_1\mathcal{B}\tilde{x}(k)-\tilde{x}'(k)\mathcal{B}'P_1M_1\nonumber\\
&\hspace{-0.8em}&\hspace{-0.8em}\times(A+B_2K_2)x(k)-\tilde{x}'(k)\mathcal{B}'P_1\mathcal{B}\tilde{x}(k).
\end{eqnarray}
Substituting (\ref{17}) from $k=s$ to $k=M$ on both sides, we have
\begin{eqnarray}\label{18}
&\hspace{-0.8em}&\hspace{-0.8em} x'(s)P_1x(s)-x'(M+1)P_1x(M+1)\nonumber\\
=&\hspace{-0.8em}&\hspace{-0.8em} J_1(s, M)+\sum\limits^{M}_{k=s}\tilde{x}'(k)\left[
                                         \begin{array}{cc}
                                           K'_1R_{11}K_1 & 0 \\
                                           0 & K'_2R_{12}K_2 \\
                                         \end{array}
                                       \right]\tilde{x}(k)\nonumber\\
&\hspace{-0.8em}&\hspace{-0.8em}-\sum\limits^{M}_{k=s}\left[
  \begin{array}{c}
    x(k) \\
    \tilde{x}(k) \\
  \end{array}
\right]'\left[
          \begin{array}{cc}
            0 & T_1\\
            T'_1 & \mathcal{B}'P_1\mathcal{B}\\
          \end{array}
        \right]
\left[
  \begin{array}{c}
    x(k) \\
    \tilde{x}(k) \\
  \end{array}
\right].
\end{eqnarray}
According to Theorem \ref{thm2}, the stability of (\ref{1-1}) means that
\begin{eqnarray*}
 \lim_{M\rightarrow \infty} x'(M+1)P_1x(M+1)=0.
\end{eqnarray*}
Thus, following from (\ref{18}) and letting $M\rightarrow \infty$, (\ref{14-1}) can be obtained exactly.

The second part is to consider the cost function of the leader  under the observer-feedback controllers (\ref{10-1})-(\ref{10-2}), that is, we will show that $J^{\star}_2(s, \infty)$ satisfies (\ref{14-2}). Following from (\ref{16}), it derives
\begin{eqnarray}\label{19-1}
&\hspace{-0.8em}&\hspace{-0.8em} x'(k)P_2x(k)-x(k+1)'P_2x(k+1)\nonumber\\
=&\hspace{-0.8em}&\hspace{-0.8em}x'(k)[P_2-(A+B_2K_2)'M'_1P_2M_1(A+B_2K_2)]x(k)\nonumber\\
&\hspace{-0.8em}&\hspace{-0.8em}-x'(k)(A+B_2K_2)'M'_1P_2\mathcal{B}\tilde{x}(k)-\tilde{x}'(k)\mathcal{B}'P_2M_1(A\nonumber\\
&\hspace{-0.8em}&\hspace{-0.8em}+B_2K_2)x(k)
-\tilde{x}'(k)\mathcal{B}'P_2\mathcal{B}\tilde{x}(k)\nonumber\\
=&\hspace{-0.8em}&\hspace{-0.8em}x'(k)[Q_2+A'S'R_{21}SA-Y'_2\Gamma^{-1}_2Y_2\nonumber\\
&\hspace{-0.8em}&\hspace{-0.8em}-A'M'_1P_2M_1B_2K_2-K'_2B'_2M'_1P_2M_1A\nonumber\\
&\hspace{-0.8em}&\hspace{-0.8em}-K'_2B'_2M'_1P_2M_1B_2K_2]x(k)-\tilde{x}'(k)\mathcal{B}'P_2\mathcal{B}\tilde{x}(k)\nonumber\\
&\hspace{-0.8em}&\hspace{-0.8em}-x'(k) (A+B_2K_2)'M'_1P_2\mathcal{B}\tilde{x}(k)\nonumber\\
&\hspace{-0.8em}&\hspace{-0.8em}-\tilde{x}'(k)\mathcal{B}'P_2M_1(A+B_2K_2)x(k),
\end{eqnarray}
where the algebraic Riccati equation (\ref{8-2}) has been used in the derivation of the last equality.
For further optimization, we make the following derivation:
\begin{eqnarray}\label{19}
&\hspace{-0.8em}&\hspace{-0.8em} x'(k)P_2x(k)-x(k+1)'P_2x(k+1)\nonumber\\
=&\hspace{-0.8em}&\hspace{-0.8em}x'(k)[Q_2+K'_1R_{21}K_1+K'_2R_{22}K_2]x(k)\nonumber\\
&\hspace{-0.8em}&\hspace{-0.8em}+x'(k)[-(A+B_2K_2)'S'R_{21}S(A+B_2K_2)\nonumber\\
&\hspace{-0.8em}&\hspace{-0.8em}-K'_2R_{22}K_2+A'S'R_{21}SA-Y'_2\Gamma^{-1}_2Y_2\nonumber\\
&\hspace{-0.8em}&\hspace{-0.8em}-A'M'_1P_2M_1B_2K_2-K'_2B'_2M'_1P_2M_1A\nonumber\\
&\hspace{-0.8em}&\hspace{-0.8em}-K'_2B'_2M'_1P_2M_1B_2K_2]x(k)-x'(k)(A+B_2K_2)'\nonumber\\
&\hspace{-0.8em}&\hspace{-0.8em}\times M'_1P_2\mathcal{B}\tilde{x}(k)-\tilde{x}'(k)\mathcal{B}'P_2M_1(A+B_2K_2)x(k)\nonumber\\
&\hspace{-0.8em}&\hspace{-0.8em}-\tilde{x}'(k)\mathcal{B}'P_2\mathcal{B}\tilde{x}(k)\nonumber\\
=&\hspace{-0.8em}&\hspace{-0.8em}x'(k)[Q_2+K'_1R_{21}K_1+K'_2R_{22}K_2]x(k)\nonumber\\
&\hspace{-0.8em}&\hspace{-0.8em}-x'(k)(A+B_2K_2)'M'_1P_2\mathcal{B}\tilde{x}(k)-\tilde{x}'(k)\mathcal{B}'P_2M_1\nonumber\\
&\hspace{-0.8em}&\hspace{-0.8em}\times(A+B_2K_2)x(k)-\tilde{x}'(k)\mathcal{B}'P_2\mathcal{B}\tilde{x}(k).
\end{eqnarray}
Substituting (\ref{19}) from $k=s$ to $k=M$ on both sides, one has
\begin{eqnarray}\label{20}
&\hspace{-0.8em}&\hspace{-0.8em} x'(s)P_2x(s)-x'(M+1)P_2x(M+1)\nonumber\\
=&\hspace{-0.8em}&\hspace{-0.8em} J_2(s, M)+\sum\limits^{M}_{k=s}\tilde{x}'(k)\left[
                                         \begin{array}{cc}
                                           K'_1R_{21}K_1 & 0 \\
                                           0 & K'_2R_{22}K_2 \\
                                         \end{array}
                                       \right]\tilde{x}(k)\nonumber\\
&\hspace{-0.8em}&\hspace{-0.8em}-\sum\limits^{M}_{k=s}\left[
  \begin{array}{c}
    x(k) \\
    \tilde{x}(k) \\
  \end{array}
\right]'\left[
          \begin{array}{cc}
            0 & T_2\\
            T'_2 & \mathcal{B}'P_1\mathcal{B}\\
          \end{array}
        \right]
\left[
  \begin{array}{c}
    x(k) \\
    \tilde{x}(k) \\
  \end{array}
\right].
\end{eqnarray}
Due to $\lim_{M\rightarrow \infty} x'(M+1)P_2x(M+1)=0$, (\ref{14-2}) can be immediately obtained by letting $M\rightarrow \infty$ in (\ref{20}).

Moreover, together with Lemma \ref{lem1}, the optimal cost functions of (\ref{13-1})-(\ref{13-2}) under the optimal feedback Stackelberg strategy are given by
\begin{eqnarray}
J^*_1(s, \infty)&\hspace{-0.8em}=&\hspace{-0.8em} x'(s)P_1x(s),\label{sy-5-1}\\
J^*_2(s, \infty)&\hspace{-0.8em}=&\hspace{-0.8em} x'(s)P_2x(s).\label{sy-5-2}
\end{eqnarray}
Together with (\ref{14-1})-(\ref{14-2}), $\delta J_1(s, \infty)$ and $\delta J_2(s, \infty)$ in (\ref{15-1})-(\ref{15-2}) are obtained.
This completes the proof.
\end{proof}

Finally, we will show the asymptotical optimal property under the observer-feedback Stackelberg strategy (\ref{10-1})-(\ref{10-2}).
\begin{theorem}\label{thm4}
Under the condition of Theorem \ref{thm2}, the optimal cost functions (\ref{14-1})-(\ref{14-2}) under the observer-feedback Stackelberg strategy (\ref{10-1})-(\ref{10-2}) are asymptotical optimal to the optimal cost functions (\ref{sy-5-1})-(\ref{sy-5-2})  under the optimal feedback Stackelberg strategy (\ref{5-1})-(\ref{5-2}), that is to say, for any $\varepsilon >0$, there exists a sufficiency large integer $N$ for $i=1, 2$ such that
\begin{eqnarray}
 \delta J_i(N, \infty)< \varepsilon.
\end{eqnarray}
\end{theorem}

\begin{proof}
Following from Theorem \ref{thm2}, there exists a stable matrix $\bar{A}$.
Thus, by \cite{Rami2001}, there exist constants $0< \lambda <1$ and $c>0$ such that
\begin{eqnarray}\label{sy-6}
\Big\|\left[
    \begin{array}{c}
      x(k) \\
      \tilde{x}(k) \\
    \end{array}
  \right]
\Big\|\leq c \lambda^k \Big\|\left[
                       \begin{array}{c}
                         x(0) \\
                         \tilde{x}(0) \\
                       \end{array}
                     \right]\Big\|.
\end{eqnarray}
In this way, one has
\begin{eqnarray}\label{sy-7}
\delta J_i(N, \infty)&\hspace{-0.8em}=&\hspace{-0.8em}\sum\limits^{\infty}_{k=s} \left[
    \begin{array}{c}
      x(k) \\
      \tilde{x}(k) \\
    \end{array}
  \right]'\left[
          \begin{array}{cc}
            0 & T_i\\
            T'_i & S_i\\
          \end{array}
        \right]
\left[
  \begin{array}{c}
    x(k) \\
    \tilde{x}(k) \\
  \end{array}
\right]\nonumber\\
\leq&\hspace{-0.8em}&\hspace{-0.8em}\sum\limits^{\infty}_{k=s}\Big\|\left[
          \begin{array}{cc}
            0 & T_i\\
            T'_i & S_i\\
          \end{array}
        \right]\Big\|\Big\|\left[
  \begin{array}{c}
    x(k) \\
    \tilde{x}(k) \\
  \end{array}
\right]\Big\|^2\nonumber\\
\leq&\hspace{-0.8em}&\hspace{-0.8em}\sum\limits^{\infty}_{k=s}\lambda^{2k}\cdot c^2\Big\|\left[
          \begin{array}{cc}
            0 & T_i\\
            T'_i & S_i\\
          \end{array}
        \right]\Big\|\Big\|\left[
                       \begin{array}{c}
                         x(0) \\
                         \tilde{x}(0) \\
                       \end{array}
                     \right]\Big\|^2\nonumber\\
<&\hspace{-0.8em}&\hspace{-0.8em}  \frac{\lambda^{2s}}{1-\lambda^2} \cdot   c^2\Big\|\left[
          \begin{array}{cc}
            0 & T_i\\
            T'_i & S_i\\
          \end{array}
        \right]\Big\|\Big\|\left[
                       \begin{array}{c}
                         x(0) \\
                         \tilde{x}(0) \\
                       \end{array}
                     \right]\Big\|^2  \nonumber\\
\doteq  &\hspace{-0.8em}&\hspace{-0.8em} \bar{c} \lambda^{2s}.
\end{eqnarray}
Since $0< \lambda <1$, thus there exists a sufficiency large integer $N$ such that for any $\varepsilon >0$, satisfying
\begin{eqnarray*}
\lambda^{2N}<\frac{1}{\bar{c}+1} \varepsilon.
\end{eqnarray*}
Combing with (\ref{sy-7}), one has
\begin{eqnarray}\label{sy-8}
\delta J_i(N, \infty) < \frac{\bar{c}}{\bar{c}+1} \varepsilon< \varepsilon.
\end{eqnarray}
That is to say,  the cost functions (\ref{14-1})-(\ref{14-2}) under the observer-feedback Stackelberg strategy (\ref{10-1})-(\ref{10-2}) are asymptotical optimal to the cost functions (\ref{sy-5-1})-(\ref{sy-5-2})  under the optimal feedback Stackelberg strategy (\ref{5-1})-(\ref{5-2}) when the integer $N$ is large enough. The proof is now completed.
\end{proof}

\section{Numerical Examples}
To show the validity of the results in Theorem \ref{thm1} to Theorem \ref{thm4}, the following example is presented.
Consider system (\ref{1-1})-(\ref{1-3}) with
\begin{eqnarray*}
  A &\hspace{-0.8em}=&\hspace{-0.8em} \left[
          \begin{array}{cc}
            1 & -0.7 \\
            1 & -0.3 \\
          \end{array}
        \right], \quad B_1=\left[
                   \begin{array}{c}
                     -5 \\
                     -1 \\
                   \end{array}
                 \right], \\
 B_2&\hspace{-0.8em}=&\hspace{-0.8em}\left[
                   \begin{array}{c}
                    0 \\
                     1 \\
                   \end{array}
                 \right], \quad   H_1=\left[
                       \begin{array}{cc}
                         1 & 0 \\
                       \end{array}
                     \right], \quad H_2=\left[
                       \begin{array}{cc}
                         0 & 1 \\
                       \end{array}
                     \right],
\end{eqnarray*}
and the associated cost functions (\ref{2-1})-(\ref{2-2}) with
\begin{eqnarray*}
Q_1 &\hspace{-0.8em}=&\hspace{-0.8em} \left[
          \begin{array}{cc}
            1 & 0 \\
            0 & 1 \\
          \end{array}
        \right], \quad Q_2 = \left[
          \begin{array}{cc}
            2 & 0 \\
            0 & 1 \\
          \end{array}
        \right],\\
R_{11}&\hspace{-0.8em}=&\hspace{-0.8em}1, \quad R_{11}=2,\quad R_{21}=0,\quad R_{22}=1.
\end{eqnarray*}

By decoupled solving the algebraic Riccati equations (\ref{8-1})-(\ref{8-2}), the feedback gains in (\ref{6-1})-(\ref{6-2}) are respectively calculated as
\begin{eqnarray*}
  K_1 &\hspace{-0.8em}=&\hspace{-0.8em} \left[
            \begin{array}{cc}
              0.2028 & -0.1374 \\
            \end{array}
          \right],\\
  K_2 &\hspace{-0.8em}=&\hspace{-0.8em} \left[
            \begin{array}{cc}
               -0.4005 &  0.0791 \\
            \end{array}
          \right].
\end{eqnarray*}
By using the LMI Toolbox in Matlab, $L_i$ ($i=1, 2$) are calculated as
\begin{eqnarray*}
L_1=\left[
                   \begin{array}{c}
                     1.2364 \\
                     0.4246 \\
                   \end{array}
                 \right], \quad L_2=\left[
                   \begin{array}{c}
                    0.0039 \\
                    0.1925 \\
                   \end{array}
                 \right],
\end{eqnarray*}
while the four eigenvalues of matrix $\mathcal{A}$ are calculated as:
\begin{eqnarray*}
  \lambda_1(\mathcal{A}) &\hspace{-0.8em}=&\hspace{-0.8em}0.1949, \quad \lambda_2(\mathcal{A})=0.6791,\\
  \lambda_3(\mathcal{A}) &\hspace{-0.8em}=&\hspace{-0.8em}\lambda_4(\mathcal{A})=0.7317,
\end{eqnarray*}
which means that $\mathcal{A}$ in (\ref{sy-11}) is sable.
In this way, following from Theorem \ref{thm1}, the state error estimation $\tilde{x}(k)$ in (\ref{12-1}) is stable, which is shown in Fig. \ref{Fig1}, where data 1 to data 4 represent the four components of vector $\tilde{x}(k)\doteq \left[
                                           \begin{array}{cccc}
                                             \tilde{x}_{11}(k) &\hspace{-0.5em} \tilde{x}_{21}(k) &\hspace{-0.5em} \tilde{x}_{31}(k) &\hspace{-0.5em} \tilde{x}_{41}(k) \\
                                           \end{array}
                                         \right]'
$. Moreover, under the observer-feedback Stackelberg strategy (\ref{10-1})-(\ref{10-2}), the state $x(k)$ in (\ref{1-1}) is also stable which can be seen in Fig. \ref{Fig2}, where data 1 and data 2 represent the two components of $x(k)\doteq \left[
                                                                                       \begin{array}{cc}
                                                                                         x_{11}(k) &\hspace{-0.5em} x_{21}(k)\\
                                                                                       \end{array}
                                                                                     \right]'$.
Finally, by analyzing Fig. \ref{Fig1} and Fig. \ref{Fig2} and selecting $N=30$ in Theorem \ref{thm4}, the asymptotical optimal property of the cost functions (\ref{14-1})-(\ref{14-2}) under the observer-feedback Stackelberg strategy (\ref{10-1})-(\ref{10-2}) is guaranteed.

%

\begin{figure}[htbp]
\centering
\includegraphics[width=3.1in, height=2.4in]{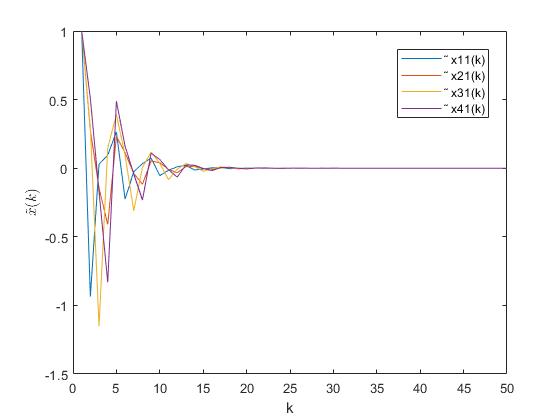}
\caption{\footnotesize Trajectory of $\tilde{x}(k)$ in (\ref{12-1}) under the observer-feedback Stackelberg strategy (\ref{10-1})-(\ref{10-2}).}
\label{Fig1}
\end{figure}

\begin{figure}[htbp]
\centering
\includegraphics[width=3.1in, height=2.4in]{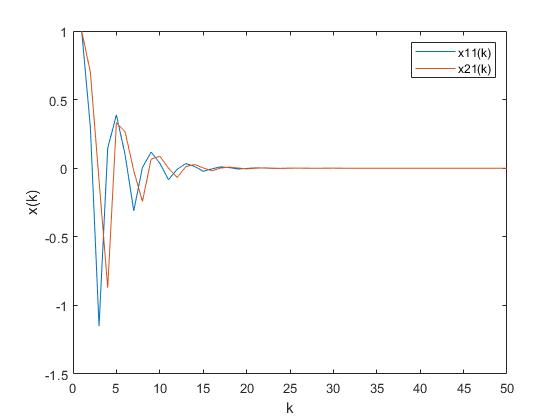}
\caption{\footnotesize Trajectory of $x(k)$ in (\ref{1-1}) under the observer-feedback Stackelberg strategy (\ref{10-1})-(\ref{10-2}).}
\label{Fig2}
\end{figure}

\section{Conclusion}
In this paper, we have considered the feedback Stackelberg strategy for two-player leader-follower game with private inputs, where the follower only shares its measurement informaiton with the leader, while none of the historical control inputs and measurement information of the leader are shared with the follower due to the hierarchical relationship.
The unavaliable access of the historical inputs for both controllers causes the main difficulty. The obstacle is overcome by designing the observers based on the informaiton structure and the observer-feedback Stackelberg strategy. Moreover, we have shown that the cost functions under the proposed observer-feedback Stackelberg strategy are asymptotical optimal to the cost functions under the optimal feedback Stackelberg strategy.

\end{document}